\documentclass[11pt]{article}
\usepackage{graphicx}
\usepackage{tikz}
\usepackage{amssymb}
\usepackage{amsmath,amsthm}
\usepackage{verbatim}
\usepackage{amsfonts}
\usepackage[T1]{fontenc}
\usepackage{multicol}
\usepackage{color}
\usepackage{hyperref}
 \hypersetup{
    colorlinks=true,
    citecolor=red,
    linkcolor=blue,
    filecolor=magenta,      
    urlcolor=cyan,
    pdftitle={Overleaf Example},
    pdfpagemode=FullScreen,
}

\newtheorem{thm}{Theorem}[section]
\newtheorem{prop}{Proposition}[section]
\newtheorem{cor}[thm]{Corollary}

\newtheorem{lem}[thm]{Lemma}
\newtheorem{conj}[thm]{Conjecture}

\newtheorem{rem}[prop]{Remark}
\usepackage{dsfont}
\newcommand\dsone{\mathds{D}}

\newcommand{\no}{\noindent}
\newcommand{\bdot}{\boldsymbol{\cdot}}

\title{An algebraic approach towards a conjecture on the Davenport constant}
\author{Naveen K. Godara\footnote{Department of Mathematics, Indian Institute of Science Education and Research Bhopal, India. Email: naveen(dot)iiserb(at)gmail.com}, Renu Joshi \footnote{Department of Mathematics, Indian Institute of Science Education and Research Bhopal, India. Email: rjoshi(dot)iiser(at)gmail.com}, and Eshita Mazumdar\footnote{Mathematical and Physical Sciences, School of Arts and Sciences, Ahmedabad University, India. Email: eshita.mazumdar(at)ahduni.edu.in}  }
\date{}
\begin{document}

\begingroup
\allowdisplaybreaks

\maketitle
\begin{abstract} 
\no For a finite group $G,$ $\mathsf{D}(G)$ is defined as the least positive integer $k$ such that for every sequence $S=g_1\bdot g_2\bdot \dotsc \bdot g_k$ of length $k$ over $G$, there exist $1 \le i_1 < i_2 <\cdots < i_m \le k $ such that  $g_{i_1}g_{i_2}\cdots g_{i_m}=1,$ where $1$ is the identity element of $G.$ The small Davenport constant $\mathsf{d}(G)$ is the maximal positive integer $k$ such that there is a sequence of length $k$ over $G$ which has no non-trivial product-one subsequence. In 2004, Dimitrov proved that $\mathsf{D}(G)\leq \mathsf{L}(G)$ for a finite $p$-group $G$, where $p$ is a prime and $\mathsf{L}(G)$ is the Loewy length of $\mathbb{F}_p[G].$ He conjectured that the equality holds for all finite $p$-groups.
In this article, we compute $\mathsf{D}(G)$ for certain classes of finite non-abelian $p$-groups, including metacyclic groups, and show that the conjecture is true by determining the precise value of $\mathsf{L}(G)$. As a consequence, we refine an upper bound on $\mathsf{d}(G)$ recently given by Qu, Li and Teeuwsen, and prove that for specific classes of groups $\mathsf{D}(G)=\mathsf{d}(G)+1$.
We also evaluate $\mathsf{D}(G)$ for finite dicyclic, semi-dihedral and other groups.

\end{abstract}


2010 AMS Classification Code: 20D60, 20D15, 11B75.
 
\section{Introduction and statements of main results}

The concept of the Davenport constant arose in the study of class groups of algebraic number fields. Specifically, if $K$ is an algebraic number field and $G$ is its class group,
the Davenport constant $\dsone(G)$ represents the maximal number of
prime ideals (counting multiplicity) that can occur in the prime decomposition of the principal ideal generated by an irreducible integer in $K$ (see \cite{Rog} for more details). For a finite group $G$ (written multiplicatively), the \textit{Davenport constant} $\dsone(G)$ is defined as the least positive integer $k$ such that any sequence of $k$ elements from $G$ (repetition allowed) contains a subsequence of $S$ whose product of terms equals the identity element when multiplied in some order.

\no For a finite abelian group $G$, $\dsone(G)$ represents the well-known Davenport constant. The question of determining $\dsone(G)$ was proposed by H. Davenport in 1966. The study of the Davenport constant is a classical and challenging problem in the fields of additive and combinatorial number theory. The exact value of $\dsone(G)$ is known only for certain types of finite abelian groups, including all finite abelian groups of rank at most two \cite{O2}, $p$-groups of any rank \cite{O1}, where $p$ is a prime, and for specific classes of abelian groups.

\no To better understand this concept, we define some essential terms. Let $G$ be a finite multiplicative group with the identity element $1$, and let $S=g_1\bdot g_2\bdot \dotsc \bdot g_{\ell}$ be a sequence of length $\ell$ over $G$. A sequence
$T=g_{i_1}\bdot g_{i_2}\bdot \dotsc \bdot g_{i_{m}}$ is said to be a {\it subsequence} of $S$
if $1\leq {i_1},\dotsc, {i_m}\leq {\ell}$ and ${i_1},\dotsc, {i_m}$ are pairwise distinct. A sequence $S=g_1\bdot g_2\bdot \dotsc \bdot g_{\ell}$ is said to be a \textit{product-one sequence} if $\prod_{i=1}^{\ell} g_{\sigma(i)}=1$ for some $\sigma \in \mathfrak{S}_{\ell},$ where $\mathfrak{S}_{\ell}$ is the symmetric group of $\ell$ symbols.

\no We remark that for a finite group $G$, $\dsone(G)=\mathsf{d}(G)+1$, where $\mathsf{d}(G)$ is the \textit{small Davenport constant}, which is the maximal positive integer $k$ such that there is a sequence of length $k$ over $G$ that does not contain any non-trivial product-one subsequence. 

\no For non-abelian groups, there are several natural extensions of the Davenport constant.
Over the past few decades, mathematicians have explored various versions of the Davenport constant for finite non-abelian groups and established their connections to distinct fields. In 1977, Olson and White \cite{O} introduced the ordered version of the Davenport constant. The {\it ordered Davenport constant} $\mathsf{D}(G)$ is defined as the least positive integer $k$ such that for every sequence $S=g_1\bdot g_2\bdot \dotsc \bdot g_{k}$ of length $k$ over $G$ there exist $1 \le i_1 < i_2 <\cdots < i_m \le k $ such that $g_{i_1}g_{i_2}\cdots g_{i_m}=1.$ We want to point out that in \cite{D}, $\mathsf{D}(G)$ is referred to as the strong Davenport constant, but in the literature, the strong Davenport constant \cite{chapman} is meant for something else. So, to avoid confusion, we call it the ordered Davenport constant. By the definitions of $\mathsf{d}(G)$ and $\mathsf{D}(G)$, it is evident that $\mathsf{d}(G)+1 \le \mathsf{D}(G)$ for any finite group $G$. In the same paper, Olson and White gave a general upper bound $\mathsf{D}(G) \le \big\lceil \frac{|G|+1}{2} \big\rceil$ for any finite non-cyclic group $G$. Therefore, it follows that $\mathsf{d}(G) \le \big\lfloor \frac{|G|}{2}\big\rfloor$ for any finite non-cyclic group $G$. Another version was introduced by Geroldinger and Grynkiewicz (see \cite{AG} and \cite{GG} for more details), known as the large Davenport constant. The \textit{large
Davenport constant} $\mathcal{D}(G)$ is defined as the maximal length of a product-one sequence that cannot be partitioned into two non-trivial product-one subsequences. It was also shown that $\mathsf{d}(G)+1\leq\mathcal{D}(G)$, with equality holds when $G$ is abelian.

\no From the above discussion, it is clear that $\mathsf{d}(G)+1 \le \min \{\mathsf{D}(G),\mathcal{D}(G)\}$ for any finite group $G$. However, there is no concrete relationship between $\mathsf{D}(G)$ and $\mathcal{D}(G).$ For example, it has been observed that if $G$ is the $\text{ non-abelian group of order } 27 \text{ of exponent } 3 $, then $\mathsf{d}(G)=6$, $\mathcal{D}(G)= 8$, and $\mathsf{D}(G)= 9$. On the other hand, when $G$ is the non-abelian group of order $27$ of exponent $9$, then $\mathsf{d}(G)= 10$, $\mathcal{D}(G)= 12$ and $\mathsf{D}(G)= 11$ (we refer to  \cite{NS}, \cite{CDS}, and \cite{D} for details). 

\no If $G$ is abelian, we have $\mathsf{d}(G)+1=\mathsf{D}(G) = \mathcal{D}(G).$ Once the question shifts to determining the precise value of these invariants for finite non-abelian groups, the problem becomes more complicated than finding the same for finite abelian groups. Recently, Qu, Li and Teeuwsen proved the following upper bound on $\mathsf{d}(G)$: 
\begin{thm}\label{QuLiTeeuwsen2022}\emph{\cite[Theorem 1.1]{QuLiTeeuwsen2022}} Let $G$ be a finite non-cyclic group, and $p$ be the smallest prime divisor of $|G|$. Then
$$\mathsf{d}(G) \le |G|/p + p - 2$$
with equality if $G$ contains a cyclic subgroup of index $p$.
\end{thm}
\no The invariants $\mathsf{d}(G)$ (hence $\dsone(G)$) and $\mathcal{D}(G)$ have been studied extensively, but not much is known about $\mathsf{D}(G)$, for more details, we refer \cite{NS,GG}. So, throughout the paper, we mainly focus on $\mathsf{D}(G)$.

\no Building upon the bound for $\mathsf{D}(G)$ provided by Olson and White, we can derive our first result, which is as follows:
\begin{thm}\label{thm1}
Let $n \ge 2 $ be an integer. Then\\
\no \emph{(i)} For the dicyclic group $Q_{4n} = \langle x,y \mid x^2 = y^n,  y^{2n} =1 , x^{-1}yx=y^{-1}\rangle,$ we have 
$$\mathsf{D}(Q_{4n})=2n+1.$$\\ 
\no \emph{(ii)}   For the semi-dihedral group $SD_{8n}= \langle x,y \mid x^2 = y^{4n}=1,  x^{-1}yx= y^{2n-1}\rangle,$ we have
    $$\mathsf{D}(SD_{8n}) = 4n+1.$$
\end{thm}

\no After Olson and White, in 2004, Dimitrov \cite{D} dealt with the ordered Davenport constant and provided an upper bound on $\mathsf{D}(G)$ for any finite $p$-group $G$. 
\no Let $p$ be a prime number, and $G$ be a finite $p$-group. Then the nilpotency index of the Jacobson radical of the modular group algebra $\mathbb{F}_p[G]$ is known as the \textit{Loewy length} of $\mathbb{F}_p[G]$, and will be denoted by $\mathsf{L}(G)$ in this paper. Dimitrov established the following result:
\begin{thm}\label{thm2}\emph{\cite{D}}
    For a prime $p$ and a finite $p$-group $G,$ we have $\mathsf{D}(G) \le \mathsf{L}(G).$
\end{thm}
\no For finite abelian $p$-groups and for the finite group of order $p^3$ with exponent $p^2$, equality holds in Theorem \ref{thm2}. Dimitrov proved the above theorem for the finite group of order $p^3$ with exponent $p$ when $p \equiv 3\ (\textrm{mod}\ 4)$. Based on these results, he conjectured the following {\cite{D}}: 
\begin{conj}\label{conj1}
    For a prime $p$ and a finite $p$-group $G,$ we have $\mathsf{D}(G)= \mathsf{L}(G).$
\end{conj}
\no In this article, we consider certain classes of finite $p$-groups (where $p$ is odd) and determine the exact values of the small Davenport constant, ordered Davenport constant, Loewy length and their relations. As a result, we prove Conjecture \ref{conj1} for these groups.

\no We consider the following groups:\\
$$G_1(\alpha,\beta,\gamma) =  (\langle c \rangle\times \langle a\rangle)\rtimes \langle b\rangle,$$ 
\no where $[a,b]=c, [a,c]=[b,c]=1,
    o(a)=p^{\alpha}, o(b)=p^{\beta}, o(c)=p^{\gamma}, \alpha, \beta, \gamma \in \mathbb{N} \;\text{with}\; \alpha \geq \beta \geq\gamma \geq 1.$
$$G_2(\alpha,\beta) = \langle a \rangle\rtimes \langle b \rangle,$$ 
\no where $[a,b]=a^{p^{\alpha-\gamma}},
    o(a)=p^{\alpha}, o(b)=p^{\beta}, o([a,b])=p^{\gamma},$
    $\alpha, \beta, \gamma \in \mathbb{N}\; \text{with}\; \alpha \geq 2\gamma, \beta  \geq \gamma \geq 1.$
$$G_3(\alpha,\beta,\sigma) =  (\langle c \rangle\times \langle a\rangle)\rtimes \langle b\rangle,$$
\no where $[a,b]=a^{p^{\alpha-\gamma}} c, 
    [c,b]= a^{-p^{2(\alpha-\gamma)}}c^{-p^{\alpha-\gamma}},$ $o(a)=p^{\alpha}, o(b)=p^{\beta}, o(c)=p^{\sigma},
\alpha, \beta, \gamma, \sigma \in \mathbb{N}$ with $\beta \ge \gamma > \sigma \geq 1, \alpha +\sigma \geq 2\gamma.$\\
\no For more details related to the above-defined groups, refer to \cite{BK}.

\no We now present the main result of this article:
\begin{thm}\label{thm6}
Let $p$ be an odd prime. Then we have the following:\\
\no  \emph{(i)} $\mathsf{L}(G_1(\alpha,\beta,\gamma))=p^{\alpha}+p^{\beta}+2p^{\gamma}-3,~~\text{and}~~\mathsf{D}(G_1(\alpha,\beta,1))=p^{\alpha}+p^{\beta}+2p-3.$\\ 
\emph{(ii)} $\mathsf{L}(G_2(\alpha,\beta))=\mathsf{D}(G_2(\alpha,\beta))=p^{\alpha}+p^{\beta}-1.$\\
\no \emph{(iii)} $\mathsf{L}(G_3(\alpha,\beta,\sigma))=p^{\alpha}+p^{\beta}+2p^{\sigma}-3,~~\text{and}~~\mathsf{D}(G_3(\alpha,\beta,1))=p^{\alpha}+p^{\beta}+2p-3.$
\end{thm}

\begin{rem}
\no \emph{(i)} In proving Theorem \emph{\ref{thm6}}, we use the structure of power subgroups and a concept of quadratic non-residue for binary forms. To the best of our knowledge, the methodology used to prove Theorem \emph{\ref{thm6}}, which modifies Dimitrov's approach, has not been previously utilized to examine the defined combinatorial invariants.\\ 
\no \emph{(ii)} For the group $G_1(\alpha,\beta,1),  
 G_2(\alpha,\beta)$ and $G_3(\alpha,\beta,1)$, it is clear that the upper bounds in Theorem \emph{\ref{thm6}} comparing to Theorem \emph{\ref{QuLiTeeuwsen2022}} are better.\\
\no \emph{(iii)} Theorem \emph{\ref{thm6}} implies that Conjecture \emph{\ref{conj1}} holds for $G_1(\alpha,\beta,1), G_2(\alpha,\beta)$, and $G_3(\alpha,\beta,1)$.\\
\no \emph{(iv)} Note that the group $G_1(1,1,1)$ has order $p^3$ and exponent $p$. This completes Dimitrov's result \emph{\cite[Theorem 3]{D}} for any odd prime $p$.

\end{rem} 
\no We obtain the following result for the metacyclic group $G_2(\alpha,\beta)$.
\begin{cor}\label{Cor-thm6} For $G=G_2(\alpha,\beta)$, we have $\mathsf{d}(G)+1=\mathsf{D}(G)=p^{\alpha}+p^{\beta}-1.$
\end{cor}

\no We also consider certain even-order non-abelian groups verifying Conjecture {\ref{conj1}}:

\begin{thm}\label{thm7}
Let $r$ be an integer. Then
$$\mathsf{D}(G)=2^{r-1}+1=\mathsf{L}(G)$$
if $G$ is one of the following groups:
 
\no \emph{(i)} $D_{2^r}=\langle x,y| x^{2}=y^{2^{r-1}}=1,x^{-1}yx=y^{-1}\rangle$ for $r \geq 3$,

\no \emph{(ii)} $Q_{2^r}=\langle x,y| x^{2}=y^{2^{r-2}},y^{2^{r-1}}=1,x^{-1}yx=y^{-1}\rangle$ for $r \geq 3,$

\no \emph{(iii)} $SD_{2^r}=\langle x,y| x^{2}=y^{2^{r-1}}=1,x^{-1}yx=y^{2^{r-2}-1}\rangle$ for $r \geq 4,$

\no \emph{(iv)} $M_{2^r}=\langle x,y|x^{2}=y^{2^{r-1}}=1,x^{-1}yx=y^{2^{r-2}+1} \rangle$ for $r \geq 4.$
\end{thm}
\begin{cor}\label{Cor-thm7}
    For an integer $r\ge 4,$ we have $\mathsf{d}(M_{2^r})+1=\mathsf{D}(M_{2^r})=2^{r-1}+1.$
\end{cor}

\no The rest of the paper is organised as follows. In Section 2, we introduce the necessary notation and recall important background results that will be used throughout the paper. In Section 3, we discuss the structure of groups and compute the exact values of their Loewy length.  In Section 4, we prove our main results, Theorem \ref{thm1}, Theorem \ref{thm6}, and Theorem \ref{thm7}. We conclude the paper with some remarks and open questions.

\section{Preliminaries}

Let $G$ be a finite group, written multiplicatively with the identity element $1$, and let $({\mathcal{F}}(G), \bdot)$ denote the free abelian monoid generated by $G$ (as a set), with the operation "$\bdot$" representing the multiplication in ${\mathcal{F}}(G)$. A sequence $S$ over $G$ is an element of ${\mathcal{F}}(G)$, typically written as $S=g_1\bdot g_2\bdot \dotsc \bdot g_{\ell}$, where the non-negative integer $\ell$ denotes the length of the sequence. The identity element $1_{\mathcal{F}(G)}$ of ${\mathcal{F}(G)}$ is called the empty or trivial sequence, which is simply the sequence having no terms (or length 0). For two sequences $S = g_1 \bdot \dotsc \bdot g_{\ell_1}$ and $T = h_1 \bdot \dotsc \bdot h_{\ell_2}$ over $G$, we write $S \bdot T := g_1 \bdot \dotsc \bdot g_{\ell_1} \bdot h_1 \bdot \dotsc \bdot h_{\ell_2}$ in $\mathcal{F}(G)$. If we denote a sequence $S$ as $g_1^{(n_1)}\bdot g_2^{(n_2)}\bdot \dotsc \bdot g_{k}^{(n_k)}$, this means $g_i$ repeats $n_i$ times in $S,$ and length of the sequence is $n_1+\ldots+n_k.$

\no It is quite useful to have a related notion in which the order of terms matters. 
 Let $(\mathcal{F}^{\ast}(G), \bdot)$ denote the free non-abelian monoid with basis $G$, whose elements are called ordered sequences over $G$, where we use the same notation "$\bdot$" for convenience. An ordered sequence $T$ over $G$ is said to be an {\it ordered subsequence} of $S=g_1\bdot g_2 \bdot g_3 \bdot \dots \bdot g_{\ell}$ $\in \mathcal{F}^{\ast}(G)$ if $T=g_{i_1}\bdot g_{i_2} \bdot g_{i_3} \bdot \dotsc \bdot g_{i_m}$ for some $1\le i_1<i_2<i_3 < \dots <i_m\le \ell$.  
  Additionally, a sequence $S=g_1\bdot g_2 \bdot g_3 \bdot \dots \bdot g_{\ell}$ is said to be \textit{product-one ordered sequence} if $\prod_{i=1}^{\ell} g_{i}=1.$

\no Throughout the paper, we fix some standard notations. For a given positive integer $n$, we denote the set $\{1,\ldots,n\}$ by $[n]$, and 
$\mathfrak{S}_{n}$ denotes the symmetric group of $n$ symbols. The commutator of elements $x, y$ in $G$ is defined as $[x,y]:=x^{-1}y^{-1}xy$. Here, the commutators are left normed, i.e.,
$[x, y, z] = [[x, y], z]$. If $H_i$ is a subgroup of $G$ for $i\in [2],$ then $[H_1,H_2]$ is also a subgroup of $G$ which is generated by $[h_1,h_2],$ where $h_i \in H_i$ for $i\in [2].$ For any $n\in \mathbb{N}$, $G^{n}:=\langle x^n \mid x\in G \rangle$ is a subgroup of $G.$ These groups form a decreasing chain of subgroups of $G$.
For a group $G$, $\gamma_1(G):=G,$ $\gamma_2(G):=[G,G]$ and $\gamma_i(G):=[\gamma_{i-1}(G), G]$ for $i\ge 3$. A finite group $G$ is said to be of nilpotency class two if $\gamma_3(G)=1$ but $ \gamma_i(G) \neq 1 \text{ for } i\in [2].$ 

\no Let $p$ be a prime number, and $G$ be a finite $p$-group. Since the modular group algebra $\mathbb{F}_p[G]$ has a unique two-sided maximal ideal, the Jacobson radical $J$ coincides with the augmentation ideal generated by $\{g-1 \mid g \in G \setminus \{1\}\}$ of $\mathbb{F}_p[G].$
 The dimensions of the Loewy factors $J^i/J^{i+1}$ have been computed by Jennings (see \cite{J} for more details) in terms of the Brauer-Jennings-Zassenhaus series (or $M$-series) $M_{i,p}(G)$ of $G$, where
 $M_{1,p}(G):=G$ and for $i\geq 2$, \[M_{i,p}(G):=[M_{i-1, p}(G), G]M_{j,p}(G)^{p},\] where $j$ is the smallest integer satisfying $jp\geq i$. For convenience, we use the notation $M_i=M_{i,p}(G)$ for all $i\geq 1,$ whenever the underlying group $G$ is understood. Clearly, each $M_i$ is a finite $p$-subgroup of $G,$ and $M_i $ is a normal subgroup of $M_{i-1}$. Through induction, it can be shown that $$G=M_1 \supseteq M_2\supseteq...\supseteq M_d\supseteq M_{d+1}=1$$ for some natural number $d$ such that $M_{d}\neq 1$. We write:
$$|M_i/M_{i+1}|=p^{e_i} \; \; \text{ for some natural number } e_i,  \text{ for all } i\in [d].$$
 \no For our use, we record the following classical result about $\mathsf{L}(G)$.
 \begin{lem}\emph{\cite[Jennings, 1941]{J}}\label{thm5} Let $G$ be a finite $p$-group, and let $\{M_i\}$ denote the $M$-series of $G$. Then\\
\no \emph{(i)} For any index $i$, the quotient $M_i/M_{i+1}$ is a finite elementary abelian $p$-group.\\
\no \emph{(ii)}  The Loewy length $\mathsf{L}(G)=1+(p-1)\sum_{i=1}^d ie_i$, where $d$ is the largest positive integer with $M_{d}\neq 1$ and $p^{e_i}$ is the order of $M_i/M_{i+1}$.
\end{lem}

\section{Power subgroups and Loewy length}
\no From this point onward, we fix $p$ as an odd prime. In the following result, we find the structure of subgroups appearing in the $M$-series in terms of power subgroups.
\begin{prop}\label{Ind} 
Let $G$ be a $p$-group of nilpotency class two. Then for all $s \geq 1$, we have 
$$M_i=
\begin{cases} 
{\gamma_{2}(G)}^{p^{s}} G^{p^{s}} &\mbox{if }~ 2p^{s-1}+1 \le i \le p^s,\\
{\gamma_{2}(G)}^{p^{s}} G^{p^{s+1}} &\mbox{if }~ p^{s}+1 \le i \le 2p^s.
\end{cases}$$
\end{prop}

\begin{proof} From the definition of the $M$-series of $G$, we have $M_1=G$ and $M_2=\gamma_{2}(G)G^{p}$. Assume that $P(s)$ and $Q(s)$ represent the equalities 
\begin{eqnarray*}
P(s) & : & M_{2p^{s-1}+1}= \cdots =M_{p^{s}}={\gamma_{2}(G)}^{p^{s}} G^{p^{s}},\\
Q(s) & : & M_{p^{s}+1}= \cdots =M_{2p^{s}}={\gamma_{2}(G)}^{p^{s}} G^{p^{s+1}},
\end{eqnarray*}
for $s \ge 1.$ We will use simultaneous induction on $s \ge 1$ to prove these. We assume that $P(s)$ and $Q(s)$ are both true (the base case $s=1$ is similar). To prove $P(s+1),$ assume that 
$$M_j={\gamma_{2}(G)}^{p^{s+1}} G^{p^{s+1}} \text{ for all } 2p^{s}+1 \le j \le k$$ for some $k< p^{s+1}.$ Now, by the definition of the $M$-series, we have 
$$M_{j+1}= [M_j,G]M^{p}_{\lceil \frac{j+1}{p} \rceil}.$$
Since $2p^{s-1}+1 \le \lceil \frac{j+1}{p} \rceil \le p^s,$ using induction hypothesis, we have 
$$M_{j+1}= [{\gamma_{2}(G)}^{p^{s+1}} G^{p^{s+1}},G] \Big({\gamma_{2}(G)}^{p^{s}} G^{p^{s}} \Big)^{p}.$$
Using the fact that $G$ has nilpotency class two, we obtain
$$M_{j+1}= [G^{p^{s+1}},G] \Big({\gamma_{2}(G)}^{p^{s}} G^{p^{s}} \Big)^{p}={\gamma_{2}(G)}^{p^{s+1}} G^{p^{s+1}}.$$
This proves $P(s+1)$. The proof of $Q(s+1)$ is similar.
\end{proof}
\no From the above proposition, it remains to find the structure of power subgroups. 
\begin{prop}{\label{power}}
Let $G$ be a finite $p$-group of nilpotency class two with $G=\langle a,b \rangle.$ Then $G^{p^s}=\langle a^{p^s}, b^{p^s}, [a,b]^{p^s} \rangle$ for all $s\geq 1$.
\end{prop}
\begin{proof}
 Since $G$ is a $p$-group of nilpotency class two, we have $G^{p^s}=\{g^{p^s}\;|\; g \in G\}$ \cite[Theorem 2.10]{F}.
 By definition, $\langle a^{p^s}, b^{p^s}, [a,b]^{p^s} \rangle \subseteq G^{p^s}$. Conversely, let $y \in G^{p^s}$, so $y=h^{p^s} \;\text{for some}\; h \in G$. Therefore,
\begin{align*}
    y & =(a^i b^j [a, b]^t)^{p^s}&&\text{(for some non-negative integers $i,j,t$)}\\
    & =a^{ip^s}(b^j[a,b]^t)^{p^s}[b^j[a,b]^t, a^i]^{p^s \choose 2}\\
    &=a^{ip^{s}}b^{jp^{s}}[a,b]^{p^{s}(t-\frac{ij(p^s -1)}{2})}&&\text{(since $\frac{(p^s -1)}{2}$ is an integer)}.
\end{align*}
This implies $y \in \langle a^{p^s}, b^{p^s}, [a,b]^{p^s} \rangle$. Thus $G^{p^s}=\langle a^{p^s}, b^{p^s}, [a,b]^{p^s} \rangle$ for all $s\geq 1$.
\end{proof}
\begin{prop}
The groups $G_1(\alpha,\beta,\gamma),G_2(\alpha,\beta)$ and $G_3(\alpha,\beta,\sigma)$ are $2$-generated $p$-groups of nilpotency class two.   
\end{prop}
\begin{proof}
It is clear from \cite{BK}.    
\end{proof}
\no To prove Theorem \ref{thm6}, we use the following crucial results:

\begin{lem}\label{UG1}
Let $p$ be an odd prime and  $G$ be a finite non-abelian $p$-group such that
 $$G =  (\langle c \rangle\times \langle a\rangle)\rtimes \langle b\rangle,$$ where $[a,b]=c, [a,c]=[b,c]=1,$ 
    $o(a)=p^{\alpha}, o(b)=p^{\beta}, o(c)=p^{\gamma}, 
    \alpha, \beta, \gamma \in \mathbb{N} \;\text{with}\; \alpha \geq \beta \geq\gamma \geq 1.$ Then $\mathsf{L}(G)=p^\alpha +p^\beta +2p^\gamma-3$.
    
\end{lem}
\begin{proof} 

Observe that ${\gamma_2 (G)}=\langle [a,b]\rangle$. Therefore, ${\gamma_2 (G)}^{p^s}=\langle [a,b]^{p^s}\rangle=\langle c^{p^s} \rangle,$ and
${\gamma_2 (G)}^{p^s} \subseteq G^{p^s} \;\text{for all}\; s\geq 1.$  We now show that ${\gamma_2 (G)}^{p^{s-1}}\cap G^{p^s}= \langle c^{p^s}\rangle \;\text{for all}\;1\leq s\leq \gamma,$ and hence 
${\gamma_2 (G)}^{p^{s-1}} \nsubseteq G^{p^s} \;\text{for all}\; 1\leq s\leq \gamma$. 

\no Note that $\langle c^{p^s}\rangle \subseteq {\gamma_2 (G)}^{p^{s-1}}\cap G^{p^s}$. 
Let $g \in {\gamma_2 (G)}^{p^{s-1}}\cap G^{p^s}$ then by Proposition \ref{power}, we can write $g=a^{ip^s} b^{jp^s} [a, b]^{tp^s}$ for some non-negative integers $i,j,t$. Also, we have $g=[a,b]^{kp^{s-1}}$ for some non-negative integer $k$. Thus, $a^{ip^s} b^{jp^s} [a, b]^{tp^s}=[a,b]^{kp^{s-1}}$, and consequently $$a^{ip^s} b^{jp^s} [a, b]^{p^{s-1}(pt-k)}=1.$$ Let us define $\bar{G}=\frac{G}{\gamma_2({G})}.$
Set $\bar{a}=a\gamma_2({G})$ and $\bar{b}=b\gamma_2({G})$. Hence, $\bar{G} \cong C_{p^\alpha}\times C_{p^\beta}$. So, we have $\bar{a}^{ip^s} \bar{b}^{jp^s}=\bar{1}$. Therefore, $ip^s \equiv 0\ (\textrm{mod}\ p^{\alpha})\; \text{and}\; jp^s \equiv 0\ (\textrm{mod}\ p^{\beta})$. Thus $g=[a, b]^{tp^s}\subseteq \langle c^{p^s}\rangle$. This shows ${\gamma_2 (G)}^{p^{s-1}}\cap G^{p^s}=\langle c^{p^s}\rangle$.

\noindent Also, notice that ${\gamma_2 (G)}^{p^{s-1}}=1\;\text{for all}\;\gamma+1\leq s\leq \alpha, \;\text{which implies}\; {\gamma_2 (G)}^{p^{s-1}} \subseteq G^{p^s}\;\text{for all}\;\gamma+1\leq s\leq \alpha$. Therefore, we have the following using Proposition \ref{Ind}:

\noindent \textbf{Case 1:} If $\alpha=\gamma\geq 1$, then 
\[M_k=\begin{cases} 
    G & \text{if} \;\; k=1\\
    \gamma_{2}(G)G^{p} & \text{if} \;\; k=2\\
G^{p}  &  \text{if} \;\;3 \leq k\leq p\\
{\gamma_{2}(G)}^{p^i} G^{p^{i+1}} & \text{if} \;\; p^i+1 \leq k\leq 2p^i, \text{ for all }i \in [\gamma-2]\\
G^{p^{i+1}} & \text{if} \;\; 2p^i+1 \leq k\leq p^{i+1},  \text{ for all } i \in [\gamma-2]\\
{\gamma_{2}(G)}^{p^{\gamma-1}} & \text{if} \;\; p^{\gamma-1}+1 \leq k\leq 2p^{\gamma-1}\\
1 & \text{if} \;\; 2p^{\gamma-1}+1 \leq k\leq p^{\gamma}.
\end{cases}\]
\noindent \textbf{Case 2:} For the cases $\alpha\geq\beta>\gamma\geq 1$ and $\alpha > \beta=\gamma\geq 1$, we obtain
\[M_k=\begin{cases} 
    G & \text{if} \;\; k=1\\
    \gamma_{2}(G)G^{p} & \text{if} \;\; k=2\\
G^{p}  &  \text{if} \;\;3 \leq k\leq p\\
{\gamma_{2}(G)}^{p^i} G^{p^{i+1}} & \text{if} \;\; p^i+1 \leq k\leq 2p^i, \text{ for all } i \in [\gamma-1]\\
G^{p^{i+1}} & \text{if} \;\; 2p^i+1 \leq k\leq p^{i+1},  \text{ for all } i \in [\gamma-1]\\
G^{p^{i+1}} & \text{if} \;\; p^i+1 \leq k\leq p^{i+1},  \text{ for all } \gamma \leq i \leq \alpha-2\\
1 & \text{if} \;\; p^{\alpha-1}+1 \leq k\leq p^{\alpha}.
\end{cases}\]

\no Using Proposition \ref{power}, we have
$|G^{p^i}: G^{p^{i+1}}|=p^3 \text{ for all } i\in [\gamma-1],$ 
$|G^{p^{i}}:G^{p^{i+1}}|=p^2 \text{ for all } \gamma \le i \le \beta-1,$ and $|G^{p^{i}}:G^{p^{i+1}}|=p \text{ for all } \beta \le i \le \alpha-1.$ Hence $|M_1:M_2|= p^2, |M_2:M_3|=p, 
|M_{p^i}:M_{p^{i}+1}|= p^2 \text{ for all } i\in [\beta -1],
|M_{p^{i}}:M_{{p^{i}} +1}|= p \text{ for all } \beta \le i \le \alpha -1, \text{ and }
|M_{2p^{i}}:M_{2p^{i}+1}|=p$ for all $i\in [\gamma -1].$ Therefore, by Lemma \ref{thm5} with $d=
   \begin{cases}
        2p^{\gamma-1} & \text{ if } \alpha=\beta=\gamma,\\
        p^{\alpha-1} & \text{ otherwise }
    \end{cases}$, we have 
$\mathsf{L}(G)=1+(p-1)\sum_{i=1}^{p^{\alpha-1}} ie_i=p^\alpha +p^\beta +2p^\gamma-3$.
\end{proof} 
\begin{lem}\label{UG2}
Let $p$ be an odd prime and $G$ be a finite non-abelian $p$-group such that
 $$G= \langle a \rangle\rtimes \langle b \rangle,$$ where $[a,b]=a^{p^{\alpha-\gamma}},
    o(a)=p^{\alpha}, o(b)=p^{\beta}, o([a,b])=p^{\gamma},$
    $\alpha, \beta, \gamma \in \mathbb{N}\; \text{with}\; \alpha \geq 2\gamma, \beta  \geq \gamma \geq 1.$ Then $\mathsf{L}(G)=p^\alpha +p^\beta -1.$

\end{lem}
\begin{proof}
Let us assume that $\alpha \ge \beta.$
As $[a,b]=a^{p^{\alpha-\gamma}},$ one can easily verifies that ${\gamma_2 (G)}^{p^{s-1}}=\langle a^{p^{\alpha -\gamma +(s-1)}} \rangle \subseteq \langle a^{p^s}\rangle \subseteq G^{p^s}$ for all $s\geq 1$. We have ${\gamma_2 (G)}^{p^s}=\langle [a,b]^{p^s}\rangle \subseteq G^{p^s}, \;\text{for all} \;s\geq 0$. We have the following using Proposition \ref{Ind}: 
\[M_k=\begin{cases} 
    G & \text{if} \;\; k=1\\
G^{p} &  \text{if} \;\;2 \leq k\leq p\\
G^{p^{i+1}}  & \text{if} \;\; p^i+1 \leq k\leq p^{i+1}, \text{ for all } i\in [\alpha-2]\\
1 & \text{if} \;\; p^{\alpha-1}+1 \leq k\leq p^{\alpha}.
\end{cases}\]
Using the relation, $[a,b]=a^{p^{\alpha-\gamma}},$ Proposition \ref{power} reduces to $G^{p^s}=\langle a^{p^s}, b^{p^s} \rangle$ for all $s\geq 1$. So, $|G^{p^s}|=p^{\alpha+\beta-2s}$ for all $s\geq 1.$ Now, we have
$|M_{p^i}:M_{p^{i}+1}|=|G^{p^i}:G^{p^{i+1}}|$ for all $0\leq i\leq \alpha-1.$ Therefore, $|M_{p^i}:M_{{p^i} +1}|= p^2$ for all $0\le i\le \beta-1$ and $|M_{p^i}:M_{{p^i} +1}|=p$ for all $\beta \le i\le \alpha-1.$ Then, by Lemma \ref{thm5} with $d=p^{\alpha-1}$, we have 
$\mathsf{L}(G)=1+(p-1)\sum_{i=1}^{p^{\alpha-1}} ie_i=p^\alpha +p^\beta -1.$ For the case $\beta > \alpha,$ the result follows in a similar way.
This completes the proof.
\end{proof}
\begin{lem}\label{UG3}
Let $p$ be an odd prime and $G$ be a finite non-abelian $p$-group such that
    $$G =  (\langle c \rangle\times \langle a\rangle)\rtimes \langle b\rangle,$$ where $ [a,b]=a^{p^{\alpha-\gamma}} c, 
    [c,b]= a^{-p^{2(\alpha-\gamma)}}c^{-p^{\alpha-\gamma}},$ $o(a)=p^{\alpha}, o(b)=p^{\beta}, o(c)=p^{\sigma},
\alpha, \beta, \gamma, \sigma \in \mathbb{N} \;\text{with}\; \beta \ge \gamma > \sigma \geq 1, \alpha +\sigma \geq 2\gamma.$ Then $\mathsf{L}(G)= p^\alpha +p^\beta +2p^\sigma-3.$
\end{lem}

\begin{proof}
  
Consider, $\alpha \ge \beta. $
Note that ${\gamma_2 (G)}^{p^s}=\langle [a,b]^{p^s}\rangle=\langle (a^{p^{\alpha-\gamma}}c)^{p^s} \rangle \subseteq G^{p^s}$ for all $s\geq 1$. We first prove that ${\gamma_2 (G)}^{p^{s-1}}\cap G^{p^s}= \langle [a,b]^{p^s}\rangle \;\text{for all}\; 1\leq s\leq \sigma$, then it follows that ${\gamma_2 (G)}^{p^{s-1}} \nsubseteq G^{p^s}$ for all $1\leq s\leq \sigma$.

\no Clearly, $\langle [a,b]^{p^s}\rangle\subseteq {\gamma_2 (G)}^{p^{s-1}}\cap\; G^{p^s}.$ On the other hand let us assume $g\in {\gamma_2 (G)}^{p^{s-1}}\cap \;G^{p^s}.$ Then $g=a^{ip^s} b^{jp^s} [a, b]^{tp^s}$ $ =[a,b]^{kp^{s-1}},$ for some non-negative integers $i,j,t,k$. So, $a^{ip^s} b^{jp^s} [a, b]^{p^{s-1}(pt-k)}=1$. Let us define 
$\bar{G}=\frac{G}{\langle a,c\rangle}.$ Set $\bar{b}= b \langle a,c \rangle.$ Therefore, $\bar{G}=\langle \bar{b} \rangle \cong C_{p^{\beta}}$. Thus, we have $\bar{b}^{jp^s}=\bar{1},$ where $\bar{1}$ is the identity element of $\bar{G}.$ Hence, $jp^s \equiv 0\ (\textrm{mod}\ p^{\beta})$. 
So, the equation  $a^{ip^s} b^{jp^s} [a, b]^{p^{s-1}(pt-k)}=1$ reduces to $a^{ip^s}[a, b]^{p^{s-1}(pt-k)}=1$. 
As $[a,b]=a^{p^{\alpha-\gamma}} c,$ we obtain  $a^{ip^{s}+p^{\alpha-\gamma+s-1}(pt-k)}c^{p^{s-1}(pt-k)}=1$. 
Since $\langle a,c\rangle \cong C_{p^\alpha}\times C_{p^\sigma}$, it follows that $ip^{s}+p^{\alpha-\gamma+s-1}(pt-k) \equiv 0\ (\textrm{mod}\ p^{\alpha})\; \text{ and }\; p^{s-1}(pt-k)\equiv 0\ (\textrm{mod}\ p^{\sigma}).$
The condition $p^{s-1}(pt-k)\equiv 0\ (\textrm{mod}\ p^{\sigma})$, implies that $p^{\sigma-s+1}\vert (pt-k)$ (as $\sigma-s+1\geq 1\; \text{ for }\; 1\leq s\leq \sigma$). Thus, $p\vert (pt-k)$ and hence $p\vert k$. As a consequence, $g=[a,b]^{kp^{s-1}}\in \langle [a,b]^{p^s}\rangle$. This shows that ${\gamma_2 (G)}^{p^{s-1}}\cap G^{p^s}=\langle [a,b]^{p^s}\rangle$.

\no Now, we prove that ${\gamma_2 (G)}^{p^{s-1}}\subseteq G^{p^s} \;\text{for all}\; \sigma+1 \leq s \leq \alpha$. From the relation $[a,b]=a^{p^{\alpha-\gamma}} c$ and using that fact that
$o(c)=p^{\sigma}$, we have ${\gamma_2 (G)}^{p^{s-1}}=\langle [a,b]^{p^{s-1}}\rangle=\langle (a^{p^{\alpha-\gamma}}c)^{p^{s-1}}\rangle =\langle a^{p^{\alpha-\gamma+s-1}}\rangle \subseteq G^{p^s} \;\text{for all}\; \sigma+1 \leq s \leq \gamma.$
Note that $|\gamma_2(G)|=p^{\gamma}$ which implies $1={\gamma_2 (G)}^{p^{s-1}}\subseteq G^{p^s}\;\text{for all}\; \gamma+1 \leq s \leq \alpha$. Therefore, ${\gamma_2 (G)}^{p^{s-1}}\subseteq G^{p^s} \;\text{for all}\; \sigma+1 \leq s \leq \alpha$.
Using Proposition \ref{Ind}, we arrive at the following computations:
\[M_k=\begin{cases} 
    G & \text{if} \;\; k=1\\
    \gamma_{2}(G)G^{p}& \text{if} \;\; k=2\\
G^{p}  &  \text{if} \;\;3 \leq k\leq p\\
{\gamma_{2}(G)}^{p^i} G^{p^{i+1}} & \text{if} \;\; p^i+1 \leq k\leq 2p^i, \text{ for all } i\in [\sigma-1]\\
G^{p^{i+1}} & \text{if} \;\; 2p^i+1 \leq k\leq p^{i+1},  \text{ for all } i\in [\sigma-1]\\
 G^{p^{i+1}}  & \text{if} \;\; p^i+1 \leq k\leq p^{i+1},  \text{ for all } \sigma \leq i \leq \alpha-2\\
1 & \text{if} \;\; p^{\alpha-1}+1 \leq k\leq p^{\alpha}.
\end{cases}\]

\no For $s\geq 1$, let $x=a^{p^s}, y=b^{p^s}$, and $z=c^{p^s}$, then $[x,y]=[a,b]^{p^{2s}}=a^{p^{\alpha-\gamma+2s}} c^{p^{2s}}= x^{p^{\alpha-\gamma+s}} z^{p^s}$, $[y,z]=[x,y]^{p^{\alpha-\gamma}}=x^{p^{2(\alpha-\gamma)+s}} z^{p^{\alpha-\gamma+s}}$, and $[x,z]=1$. Therefore, one can deduce that
$G^{p^s}\cong  (\langle x\rangle \times \langle z \rangle)\rtimes \langle y \rangle, 
\;\text{where}\; [x,y]=x^{p^{\alpha-\gamma+s}} z^{p^s}, [y,z]=x^{p^{2(\alpha-\gamma)+s}} z^{p^{\alpha-\gamma+s}}, o(x)=p^{\alpha-s}, o(y)=p^{\beta-s}, o(z)=p^{\sigma-s}$, and hence $|G^{p^s}|= p^{\alpha +\beta +\sigma - 3s}$ for all $s\geq 1$.
Thus, we obtain
$|G^{p^i}: G^{p^{i+1}}|=p^3 \text{ for all } i\in [\sigma-1],$
$|G^{p^{i}}:G^{p^{i+1}}|=p^2 \text{ for all } \sigma \le i\le \beta-1,$ and $|G^{p^{i}}:G^{p^{i+1}}|=p \text{ for all } \beta \le i \le \alpha-1.$ Therefore, $|M_1:M_2|=p^2, |M_2:M_3|=p,$
$|M_{p^i}:M_{p^i+1}|=p^2 \text{ for all }i\in [\beta-1],$
$|M_{p^i}:M_{p^i+1}|=p \text{ for all } \beta \le i \le \alpha-1,$ and $|M_{2p^i}:M_{2p^i+1}|=p \text{ for all }i\in [\sigma-1].$
Then, by Lemma \ref{thm5} with $d=p^{\alpha-1}$, we have $\mathsf{L}(G)=1+(p-1)\sum_{i=1}^{p^{\alpha-1}} ie_i=p^\alpha +p^\beta +2p^\sigma-3$. 
For the case $\beta > \alpha,$, the result follows similarly.
\end{proof}
\section{Proofs of the main theorems}

\subsection*{Proof of Theorem \ref{thm1}}    \no (i) Recall that $Q_{4n} = \langle x,y | x^2 = y^n,  y^{2n} =1 , x^{-1}yx=y^{-1}\rangle$ for $n \ge 2. $ From \cite{O}, we know that $\mathsf{D}(G)\le \bigg\lceil \frac{|G|+1}{2} \bigg\rceil,$ i.e., $\mathsf{D}(Q_{4n}) \le 2n+1.$ On the other hand, 
    consider the sequence $S= y^{(2n-1)}\bdot x$ of length $2n$ over $Q_{4n}$. Clearly, this sequence does not contain any non-trivial product-one ordered subsequence. Therefore, $\mathsf{D}(Q_{4n}) \ge 2n+1.$\\
    
    \no (ii) By \cite{O}, for the group $SD_{8n} = \langle x,y | x^2 = y^{4n}=1,  x^{-1}yx= y^{2n-1}\rangle$ with $n \ge 2, $ we have $\mathsf{D}(SD_{8n}) \le 4n+1.$ On the other hand, 
    consider the sequence $S= y^{(4n-1)} \bdot x$ of length $4n$ over $SD_{8n}$. This sequence does not have any product-one ordered subsequence, so we conclude that $\mathsf{D}(SD_{8n}) \ge 4n+1.$ \qed

\subsection*{Proof of Theorem \ref{thm6}}

\no (i) Let $G = G_1(\alpha,\beta,\gamma).$ By Lemma \ref{UG1}, we have $ \mathsf{L}(G)= p^\alpha +p^\beta +2p^\gamma-3$, along with Theorem \ref{thm2}, we obtain $\mathsf{D}(G)\le p^\alpha +p^\beta +2p^\gamma-3.$ So, it suffices to prove the lower bound for $\mathsf{D}(G)$.

\begin{itemize}
    \item 
For $p \equiv 3\ (\textrm{mod}\ 4)$, consider the sequence 
$S = k ^{(p^{\alpha}-1)}\bdot \ell^ {(p^{\beta}-1)} \bdot m^{(p^{\gamma}-1)}\bdot n^{(p^{\gamma}-1)}$ of length $p^\alpha +p^\beta +2p^\gamma-4$ over $G,$ where 
$k={a^{-1}bc^{\frac{1}{2}}}, \ell=  b^{-1}, m= a, n= a^{2}b^{-1}c$. We claim that $S$ has no non-trivial product-one ordered subsequence. If possible, assume that there exists an ordered subsequence $T= k^{(x)}\bdot \ell^{(y)} \bdot m^{(z)} \bdot n^{(w)}$ of $S,$ where $0\le x \le (p^{\alpha}-1), 0\le y \le (p^{\beta}-1), 0\le z, w\le (p^{\gamma}-1)$ not all zero such that
 $k^x \ell^y m^z n^w = 1.$ Then we have the following system of equations: 
\begin{align*}
  -x+z+2w &\equiv 0\ (\textrm{mod}\ p^{\alpha})\\ 
  x-y-w & \equiv 0\ (\textrm{mod}\ p^{\beta}) \\ 
  -2(x-y)(z+2w)+(x^2+2w^2)&\equiv 0\ (\textrm{mod}\ p^{\gamma}).
\end{align*}

For $\gamma =1,$ expressing $x$ and $y$ from the first two equations and substituting in the third, we obtain the quadratic form $z^2+2zw+2w^2$ with discriminant $-4$, which is quadratic non-residue modulo $p.$ This implies $z=w=0$. Therefore, $x=y=z=w=0$, a contradiction.

\item Let $p \equiv 1\ (\textrm{mod}\ 4)$. Let $n_{p}$ denote the least quadratic non-residue for $p.$ Then $n_p$ will be a prime number and $n_{p}<\sqrt{p}+1.$ Set $q=n_{p}$ and consider the sequence $S = k ^{(p^{\alpha}-1)} \bdot \ell^ {(p^{\beta}-1)} \bdot m^{(p^{\gamma}-1)} \bdot n^{(p^{\gamma}-1)}$ over $G,$ where $k= a^{-1}bc^{\frac{1}{2}}, \ell= b^{-1}, m= ab^{q}c^{-\frac{q}{2}}, n= a.$ If possible, assume that there exists an ordered subsequence $T= k^{(x)} \bdot \ell^{(y)}\bdot m^{(z)} \bdot n^{(w)}$ of $S,$ where $0\le x \le (p^{\alpha}-1), 0\le y \le (p^{\beta}-1), 0\le z, w\le (p^{\gamma}-1)$ not all zero such that
 $k^x \ell^y m^z n^w = 1.$ Then we have the following system of equations:
\begin{align*}
  -x+z+w &\equiv 0\ (\textrm{mod}\ p^{\alpha})\\ 
  x-y+qz & \equiv 0\ (\textrm{mod}\ p^{\beta}) \\ 
  -2(x-y)(z+w)-2qzw+x^2-qz^2&\equiv 0\ (\textrm{mod}\ p^{\gamma}).
\end{align*}
For $\gamma=1,$ we obtain the quadratic form $(q+1)z^2+2zw+w^2$ with discriminant $-4q$, which is quadratic non-residue modulo $p $. We again have $z=w=0$, hence a contradiction.

\end{itemize}
This completes the proof of (i).\\

 \no (ii) For $G=G_2(\alpha,\beta),$ the Loewy length $\mathsf{L}(G)= p^\alpha +p^\beta -1$ follows from Lemma \ref{UG2}. By Theorem \ref{thm2}, we have $\mathsf{D}(G)\le p^\alpha +p^\beta -1.$ On the other hand, the sequence $S= a^{(p^{\alpha}-1)} \bdot b^{(p^{\beta}-1)}$ does not contain any non-trivial product-one subsequence, so $p^\alpha +p^\beta -1\leq \mathsf{D}(G)$. This completes the proof of (ii).\\

\no (iii) Let $G=G_3(\alpha,\beta,\sigma).$ From Lemma \ref{UG3}, we know that the Loewy length $\mathsf{L}(G)= p^\alpha +p^\beta +2p^\sigma-3$. According to Theorem \ref{thm2}, it follows that $\mathsf{D}(G)\le p^\alpha +p^\beta +2p^\sigma-3.$ Now, we find a lower bound for $\mathsf{D}(G)$. This calculation will be split into two cases:

\begin{itemize}
    \item For $p\equiv 3\ (\textrm{mod}\ 4)$, consider the sequence 
$S = k ^{(p^{\alpha}-1)} \bdot \ell^ {(p^{\beta}-1)} \bdot m^{(p^{\sigma}-1)} \bdot n^{(p^{\sigma}-1)}$ of length $p^\alpha +p^\beta +2p^\sigma-4$ over $G,$ where $k=a^{-1}, \ell=  b, m= ab[a,b]^{-\frac{1}{2}}, n= a^{2}b[a,b]^{-1}.$  Suppose $S$ has a non-trivial product-one ordered subsequence. Then $k^x \ell^y m^z n^w = 1$ for $0\le x \le (p^{\alpha}-1), 0\le y \le (p^{\beta}-1), 0\le z, w\le (p^{\sigma}-1)$ with $(x,y,z,w) \neq (0,0,0,0)$. In this case, we obtain the following system of equations: 
\begin{align*}
  -x+z+2w+\frac{1}{2}\times p^{\alpha-\gamma}\Big[-2y(z+2w)-4zw-(z^2+2w^2)\Big] &\equiv 0\ (\textrm{mod}\ p^{\alpha})\\ 
  y+z+w & \equiv 0\ (\textrm{mod}\ p^{\beta}) \\ 
  -2y(z+2w)-4zw-(z^2+2w^2)&\equiv 0\ (\textrm{mod}\ p^{\sigma}).
\end{align*}

 For $\sigma=1$, by substituting the value of $y$ from the second equation into the third equation, we arrive at the quadratic form
$z^2+2zw+2w^2\equiv 0\ (\textrm{mod}\ p),$ which has discriminant $-4$, a quadratic non-residue modulo $p.$  This implies $z=w=0$. Consequently, we obtain $x=y=z=w=0$, leading to a contradiction.

\item For $p\equiv 1\ (\textrm{mod}\ 4)$, let $q=n_p$ denote the least quadratic non-residue modulo $p.$ Consider the sequence 
$S = k ^{(p^{\alpha}-1)} \bdot \ell^ {(p^{\beta}-1)} \bdot m^{(p^{\sigma}-1)} \bdot n^{(p^{\sigma}-1)},$ where $k=a^{-1}, \ell=  b,$ 
$ m= ab^{(q+1)}[a,b]^{-\frac{(q+1)}{2}},$ $ n= ab[a,b]^{-\frac{1}{2}}.$ If possible, assume that there exists an ordered subsequence $T= k^{(x)} \bdot \ell^{(y)} \bdot m^{(z)}\bdot n^{(w)}$ of $S,$ where $0\le x \le (p^{\alpha}-1), 0\le y \le (p^{\beta}-1), 0\le z, w\le (p^{\sigma}-1)$ with $(x,y,z,w) \neq (0,0,0,0)$ such that
 $k^x \ell^y m^z n^w = 1.$ Then we have the following system of equations: 
\begin{align*}
  -x+z+w+\frac{1}{2}\times p^{\alpha-\gamma}\Big[-2y(z+w)-2(q+1)zw-(q+1)z^2-w^2\Big] &\equiv 0\ (\textrm{mod}\ p^{\alpha})\\ 
  y+(q+1)z+w & \equiv 0\ (\textrm{mod}\ p^{\beta}) \\ 
  -2y(z+w)-2(q+1)zw-(q+1)z^2-w^2&\equiv 0\ (\textrm{mod}\ p^{\sigma}).
\end{align*}
For $\sigma=1,$ we obtain the quadratic form $(q+1)z^2+2zw+w^2$ having discriminant $-4q,$ a quadratic non-residue modulo $p.$ So, we get $x=y=z=w=0,$ a contradiction. \qed
\end{itemize}

\subsection*{Proof of Corollary \ref{Cor-thm6}}
It follows from  the proof of the second part of Theorem \ref{thm6}.\qed

\subsection*{Proof of Theorem \ref{thm7}}

\no Let $G$ be one of the groups from $\{D_{2^r},SD_{2^r},Q_{2^r},M_{2^r}\}.$ Then we have $\mathsf{L}(G)=2^{r-1}+1$ by \cite[Theorem 1.6]{K}.  Let $S=y^{(2^{r-1}-1)}\bdot x$ be a sequence of length $2^{r-1},$ which has no non-trivial product-one subsequence over the group $G$. Hence, we have $2^{r-1}+1 \leq \mathsf{D}(G)$. As a result of Theorem \ref{thm2}, we have $\mathsf{D}(G)=2^{r-1}+1.$ Hence, we have the full truth value of Conjecture \ref{conj1} for these groups.\qed

\subsection*{Proof of Corollary \ref{Cor-thm7}}
It follows from the proof of Theorem \ref{thm7}. \qed

\section{Concluding remarks}
\begin{itemize}
    \item 
For every finite group $G,$ $\mathsf{E}(G)$ is defined as the least integer $k$ such that for every sequence $S=x_1 \bdot x_2 \bdot\dotsc \bdot x_k$ of length $k$, there exist $|G|$ indices $1 \le i_1 < i_2 <\cdots < i_{|G|} \le k $ such that  $\prod_{j=1}^{|G|} x_{i_{\sigma(j)}}=1$ holds for $\sigma = id,$ identity element of $\mathfrak{S}_{|G|}.$ Finding out the precise value of this invariant is interesting as the length of the required product-one ordered subsequence is restricted. For a finite abelian group $G$, from \cite{G} we know that  $\mathsf{E}(G)=\mathsf{D}(G)+|G|-1.$ One may ask whether a similar result also holds for finite non-abelian
groups. For a finite non-abelian group $G,$ there exists a sequence $T$ of length $\mathsf{D}(G)-1,$ which does not have a product-one ordered subsequence. If we consider the sequence $T$ append with the sequence $1^{(|G|-1)}$ of $G,$ then the new sequence does not have a product-one ordered subsequence of length $|G|.$ This conclude that $\mathsf{E}(G) \ge \mathsf{D}(G)+|G|-1.$ We believe the answer to the above question is affirmative and we conjecture the following:\\
\begin{conj}\label{conj2}
    For any finite group $G,$ $\mathsf{E}(G)=\mathsf{D}(G)+|G|-1.$
\end{conj}
\item Analogous to the weighted Davenport constant for finite abelian groups, one can also think of defining the weighted $\mathsf{D}(G)$ in the following manner: For every finite group $G$ with $\exp(G)=n$, let $A (\ne \phi) \subseteq [n-1].$ Then $A$-weighted $\mathsf{D}(G)$, denoted by $\mathsf{D}_A(G)$, is defined to be   
the least integer $k$ such that for every sequence $S=x_1 \bdot x_2 \bdot\dotsc \bdot x_k$ of length $k$, there exists $m (\in \mathbb{N})$ indices $i_1 < i_2 <\cdots < i_{m} $ such that  $\prod_{j=1}^{m} x^{a_{i_{\sigma(j)}}}_{i_{\sigma(j)}}=1$ holds for $\sigma = id,$ identity element of $\mathfrak{S}_{m}$ and for some $a_{i_{\sigma(j)}} \in A.$ It is worthy to study the behaviour of $\mathsf{D}_A(G).$ Analogous to Dimitrov \cite{D}, one may ask what will be the upper bound for $\mathsf{D}_A(G)$ so that for $A=\{1\},$ the upper bound of $\mathsf{D}(G)$ coincides with $\mathsf{L}(G).$ In \cite{BM1} and \cite{BM2} the authors had defined the following extremal problem: For finite abelian group $G$ with $\exp(G)=n,$ $A (\neq \phi)\subseteq [n-1],$ and $\dsone_A(G)$ being the $A$-weighted Davenport constant,
\begin{eqnarray*}\label{fdG_def} f^{(\dsone)}_G(k)&:=&\min\{|A|: \emptyset \neq A\subseteq[n-1]\textrm{\ satisfies\ }{\dsone}_A(G)\le k\},\\ 
                                         &:=&\infty\textrm{\ 
                                         if\ there\ is\ no\ 
                                         such\ }A.\end{eqnarray*}

Now, if we modify the definition for a finite non-abelian group $G$ and consider $\mathsf{D}_A(G)$ instead of $\dsone_A(G)$, then it is interesting to study the above extremal invariant for $2$-generated $p$-group of nilpotency class two.
\end{itemize}
\section{Acknowledgement}
The research of the second author is supported by the PMRF Fellowship (PMRF ID: 0400422), Ministry of Education (India). The research of the third author is supported by the ANRF-Core Research Grant (File No. CRG/2023/002698).

\endgroup
 
\end{document}